\documentclass[preprint]{elsarticle}
\usepackage{amsmath,amsthm,amsfonts}
\newcommand{\Hilb}{\mathcal{H}}
\newcommand{\iprod}[1]{\langle#1\rangle}
\newcommand{\Dt}{\Delta t}
\newcommand{\Utilde}{\widetilde U}
\newcommand{\betatilde}{\tilde\beta}
\newcommand{\Lapl}{\mathcal{L}}
\newcommand{\Li}{\operatorname{Li}}
\newcommand{\C}{\mathbb{C}}
\newcommand{\res}{\operatorname*{res}}
\newcommand{\Contour}{\mathcal{C}}
\newproof{plainproof}{Proof}
\newtheorem{theorem}{Theorem}

\newtheorem{lemma}[theorem]{Lemma}
\begin{document}

\begin{frontmatter}

\title{Time-stepping error bounds for fractional diffusion problems 
with non-smooth initial data}

\author[unsw]{William McLean\fnref{arc,visit}}
\address[unsw]{School of Mathematics and Statistics, 
The University of New South Wales, Sydney 2052, Australia}
\ead{w.mclean@unsw.edu.au}

\author[kfupm]{Kassem Mustapha\fnref{arc}}
\address[kfupm]{Department of Mathematics and Statistics,
King Fahd University of Petroleum and Minerals,
Dhahran, 31261, Saudi Arabia}
\ead{kassem@kfupm.edu.sa}

\fntext[arc]{We thank the Australian Research Council and the 
summer visit program of the KFUPM for their financial support of this 
project.}
\fntext[visit]{The first author completed part of this research
during an extended visit to the University of Otago, Dunedin.} 

\begin{abstract}
We apply the piecewise constant, discontinuous Galerkin method to
discretize a fractional diffusion equation with respect to time. 
Using Laplace transform techniques, we show that the method is first
order accurate at the $n$th time level~$t_n$, but the error bound
includes a factor $t_n^{-1}$ if we assume no smoothness of the initial
data.  We also show that for smoother initial data the growth in the
error bound as~$t_n$ decreases is milder, and in some cases absent
altogether. Our error bounds generalize known results for the
classical heat equation and are illustrated for a model problem.
\end{abstract}

\end{frontmatter}
\section{Introduction}
Consider an initial-value problem for an abstract, time-fractional 
diffusion
equation~\cite[p.~84]{KlafterSokolov2011}
\begin{equation}\label{eq: ivp}
\partial_tu+\partial_t^{1-\nu}Au=0
	\quad\text{for $t>0$,}
	\quad\text{with $u(0)=u_0$ and $0<\nu<1$.}
\end{equation}
Here, we think of the solution~$u$ as a function 
from~$[0,\infty)$ to a Hilbert space~$\Hilb$, with 
$\partial_tu=u'(t)$ the usual derivative with respect to~$t$, and with
\[
\partial_t^{1-\nu}u(t)=\frac{\partial}{\partial t}\int_0^t
  \frac{(t-s)^{\nu-1}}{\Gamma(\nu)}\,u(s)\,ds
\]
the Riemann--Liouville fractional derviative of order~$1-\nu$.
The linear operator~$A$ is assumed to be self-adjoint, 
positive-semidefinite and densely defined 
in~$\Hilb$, with a complete orthonormal
eigensystem $\phi_1$, $\phi_2$, $\phi_3$, \dots.  We further assume
that the eigenvalues of~$A$ tend to infinity.  Thus,
\[
A\phi_m=\lambda_m\phi_m,\quad
\iprod{\phi_m,\phi_n}=\delta_{mn},\quad
0\le\lambda_1\le\lambda_2\le\lambda_3\le\cdots,
\]
where $\iprod{u,v}$ is the inner product in~$\Hilb$; the corresponding
norm in~$\Hilb$ is denoted by~$\|u\|=\sqrt{\iprod{u,u}}$.  In
particular, we may take $Au=-\nabla^2u$ and $\Hilb=L_2(\Omega)$ for a
bounded spatial domain~$\Omega$, with~$u$ subject to homogeneous 
Dirichlet or Neumann boundary conditions on~$\partial\Omega$.  
Our problem~\eqref{eq: ivp} then reduces to the classical heat 
equation when~$\nu\to1$.

Many authors have studied techniques for the time
discretization of~\eqref{eq: ivp}, but obtaining sharp error bounds
has proved challenging.  In studies of explicit and implicit finite
difference schemes~\cite{ChenLiuAnhTurner2012,Cui2009,
LanglandsHenry2005,Mustapha2011,QuintanaYuste2013,ZhangSunLiao2014}
the error analyses typically assume that the solution~$u(t)$ is
sufficiently smooth, including at~$t=0$, which amounts to imposing 
compatibility conditions on the initial data and source term.  In our 
earlier work on discontinuous Galerkin (DG)
time-stepping~\cite{McLeanMustapha2009,MustaphaMcLean2012,
MustaphaMcLean2013}, we permitted more realistic behaviour, allowing 
the derivatives of~$u(t)$ to be unbounded as~$t\to0$, but were seeking 
error bounds that are uniform in~$t$ using variable time steps. 
In the present work, we again consider a piecewise-constant DG scheme 
but with a completely different method of analysis that leads to 
sharp error bounds even for
non-smooth initial data, at the cost of requiring a constant time 
step~$\Dt$. Our previous analysis~\cite[Theorem~5]{McLeanMustapha2009} 
of the scheme~\eqref{eq: DG}, in conjunction with relevant
estimates~\cite{McLean2010} of the derivatives of~$u$, shows, in
the special case of uniform time steps, only the sub-optimal error
bound
\begin{equation}\label{eq: suboptimal}
\|U^n-u(t_n)\|\le C\Dt^{r\nu}\|A^ru_0\|
  \quad\text{for $0\le r<1/\nu$.}
\end{equation}
In our main result, we substantially improve on~\eqref{eq: suboptimal} 
by showing that
\begin{equation}\label{eq: main result}
\|U^n-u(t_n)\|\le Ct_n^{r\nu-1}\Dt\|A^ru_0\|
  \quad\text{for $0\le r\le\min(2,1/\nu)$.}
\end{equation}
Thus, for a general $u_0\in\Hilb$ the error is of
order~$t_n^{-1}\Dt$ at~$t=t_n$, so the method is first-order accurate
but the error bound includes a factor~$t_n^{-1}$ that grows if~$t_n$
approaches zero, until at~$t=t_1$ the bound is of 
order~$t_1^{-1}\Dt=1$.  However, if $1/2\le\nu<1$ and $u_0$ is 
smooth enough to belong to~$D(A^{1/\nu})$, the domain 
of~$A^{1/\nu}$, then the error is of order~$\Dt$, uniformly in~$t_n$. 
For~$0<\nu\le1/2$, no matter how smooth~$u_0$ 
a factor~$t_n^{2r-1}$ is present. To the best of our knowledge, only 
Cuesta et~al.~\cite{CuestaLubichPalencia2006} and McLean and 
Thom\'ee~\cite[Theorem~3.1]{McLeanThomee2010} have hitherto 
investigated the time discretization of~\eqref{eq: ivp} for the
interesting case when the initial data might not be regular, the 
former using a finite difference-convolution quadrature 
scheme and the latter a method based on numerical inversion of the 
Laplace transform.

In the present work, we do not discuss the spatial discretization 
of~\eqref{eq: ivp}.  By contrast, Jin, Lazarov and 
Zhou~\cite{JinLazarovZhou2013} applied a piecewise linear finite 
element method using a quasi-uniform partition of~$\Omega$ into 
elements with maximum diameter~$h$, but with no time discretization.
They worked with an equivalent formulation of the
fractional diffusion problem,
\begin{equation}\label{eq: Lazarov}
\partial_{t,\text{C}}^\nu u-\nabla^2u=0\quad
\text{for $x\in\Omega$ and $0<t\le T$,}
\end{equation}
where $\partial_{t,\text{C}}$ denotes the Caputo fractional 
derivative, and proved
\cite[Theorems 3.5~and 3.7]{JinLazarovZhou2013} that, for an
appropriate choice of~$u_h(0)$,
\[
\|u_h(t)-u(t)\|+h\|\nabla(u_h-u)\|\le Ct^{\nu(r-1)}\times
  \begin{cases}
  h^2\ell_h\|A^ru_0\|,&r\in\{0,1/2\},\\
  h^2\|A^ru_0\|,&r=1,
  \end{cases}
\]
where $\ell_h=\max(1,\log h^{-1})$.  These estimates for the spatial
error complement our bounds for the error in a time discretization.

For a fixed step size~$\Dt>0$, we put~$t_n=n\Dt$ and define a 
piecewise-constant approximation~$U(t)\approx u(t)$ by
applying the DG 
method~\cite{McLeanMustapha2009,McLeanThomeeWahlbin1996}, 
\begin{equation}\label{eq: DG}
U^n-U^{n-1}+\int_{t_{n-1}}^{t_n}\partial_t^{1-\nu}AU(t)\,dt=0
	\quad\text{for $n\ge1$, with $U^0=u_0$,}
\end{equation}
where $U^n=U(t_n^-)=\lim_{t\to t_n^-}U(t)$ denotes the
one-sided limit from below at the $n$th time level.  Thus,
$U(t)=U^n$ for~$t_{n-1}<t\le t_n$. Since we do not 
consider any spatial discretization, $U$ is a semidiscrete 
solution with values in~$\Hilb$. A short calculation reveals
that
\[
\int_{t_{n-1}}^{t_n}\partial_t^{1-\nu}AU(t)\,dt=\Dt^\nu\sum_{j=1}^n
	\beta_{n-j}AU^j,
\]
with
\[
\beta_0=\Dt^{-\nu}\int_{t_{n-1}}^{t_n}
	\frac{(t_n-t)^{\nu-1}}{\Gamma(\nu)}\,dt
	=\frac{1}{\Gamma(1+\nu)}
\]
and, for~$j\ge1$,
\[
\beta_j=\Dt^{-\nu}\int_{t_{n-j-1}}^{t_{n-j}}
	\frac{(t_n-t)^{\nu-1}-(t_{n-1}-t)^{\nu-1}}{\Gamma(\nu)}\,dt
	=\frac{(j+1)^\nu-2j^\nu+(j-1)^\nu}{\Gamma(1+\nu)}.
\]
Thus, by solving the recurrence relation
\begin{equation}\label{eq: recurrence}
(I+\beta_0\Dt^\nu A)U^n=U^{n-1}-\Dt^\nu\sum_{j=1}^{n-1}
	\beta_{n-j}AU^j
\end{equation}
for $n=1$, $2$, $3$, \dots we may compute $U^1$,
$U^2$, $U^3$, \dots.

In the classical limit as~$\nu\to1$, the fractional-order 
equation~\eqref{eq: ivp} reduces to an abstract heat equation,
\begin{equation}\label{eq: heat eqn}
\partial_tu+Au=0\quad\text{for $t>0$,}\quad\text{with $u(0)=u_0$,}
\end{equation}
and the time-stepping DG method~\eqref{eq: DG} reduces to the implicit
Euler scheme
\begin{equation}\label{eq: implicit Euler}
\frac{U^n-U^{n-1}}{\Dt}+AU^n=0,
\end{equation}
for which the following error bound 
holds~\cite[Theorems 7.1~and 7.2]{Thomee1997}:
\begin{equation}\label{eq: error heat}
\|U^n-u(t_n)\|\le Ct_n^{r-1}\Dt\|A^ru_0\|\quad
	\text{for $n=1$, $2$, $3$, \dots and $0\le r\le1$.}
\end{equation}
This result is just the limiting case 
as~$\nu\to1$ of our error estimate~\eqref{eq: main result} for the 
fractional diffusion equation.

For any real~$r\ge0$, we can characterize $D(A^r)$ in terms of the 
generalized Fourier coefficients in an eigenfunction expansion,
\[
v=\sum_{m=1}^\infty v_m\phi_m,\quad v_m=\iprod{v,\phi_m}.
\]
Indeed, $v\in\Hilb$ belongs to~$D(A^r)$ if and only if
\begin{equation}\label{eq: ||A^rv||}
\|A^rv\|^2=\sum_{m=1}^\infty\lambda_m^{2r}v_m^2<\infty,
\end{equation}
in which case the series $A^rv=\sum_{m=1}^\infty\lambda_m^rv_m\phi_m$
converges in~$\Hilb$.  Thus (recalling our assumption that 
$\lambda_m\to\infty$) 
the larger the value of~$r$ such that~$v\in D(A^r)$, the faster the 
Fourier coefficients~$v_m$ decay as~$m\to\infty$ and the ``smoother'' 
$v$ is.  When~$\Hilb=L_2(\Omega)$ the functions 
in~$D(A^r)$ may have to satisfy compatibility conditions 
on~$\partial\Omega$; see Thom\'ee~\cite[Lemma~3.1]{Thomee1997}
or \cite[Section~3]{McLean2010}. In particular, an infinitely 
differentiable function will be somewhat ``non-smooth'' if it fails to 
satisfy the boundary conditions of our problem.

We note that, for a given~$u_0$, the exact solution~$u$ is less 
smooth than is the case for the classical heat equation.  To see why, 
consider the Fourier expansion
\begin{equation}\label{eq: u(t) Fourier}
u(t)=\sum_{m=1}^\infty u_m(t)\phi_m,\qquad 
u_m(t)=\iprod{u(t),\phi_m},
\end{equation}
and put $u_{0m}=\iprod{u_0,\phi_m}$.
The Fourier coefficients~$u_m(t)$ satisfy the initial-value problem
\begin{equation}\label{eq: ivp um}
u_m'+\lambda_m\partial_t^{1-\nu}u_m=0,\quad
        \text{for~$t>0$, with $u_m(0)=u_{0m}$,}
\end{equation}
so that, as is well known~\cite{McLean2010},
$u_m(t)=E_\nu(-\lambda_m t^\nu)u_{0m}$ where $E_\nu$ denotes the 
Mittag--Leffler function. Since $E_\nu(-s)=O(s^{-1})$ decays slowly 
as~$s\to\infty$ for~$0<\nu<1$, in comparison to~$E_1(-s)=e^{-s}$, the 
high frequency modes of the solution are not damped as rapidly as in 
the classical case~$\nu=1$.

Section~\ref{sec: integral rep} uses Laplace transform techniques to 
derive integral representations for the Fourier coefficients 
$U^n_m=\iprod{U^n,\phi_m}$~and $u_m(t_n)=\iprod{u(t_n),\phi_m}$.  We 
show that $U^n_m-u_m(t_n)=\delta^n(\mu)u_{0m}$, where $\delta^n(\mu)$ 
is given by an explicit but complicated integral; thus,
the error has a Fourier expansion of the form
\begin{equation}\label{eq: Un-u(tn)}
U^n-u(t_n)=\sum_{m=1}^\infty\delta^n(\lambda_m\Dt^\nu)u_{0m}\phi_m,
\quad u_{0m}=\iprod{u_0,\phi_m}.
\end{equation}
Theorem~\ref{thm: delta^n(mu)} states a key estimate 
for~$\delta^n(\mu)$, but to avoid a lengthy digression the proof is 
relegated to Section~\ref{sec: technical}.

The main result~\eqref{eq: main result} of the paper is established 
in Section~\ref{sec: error}, where we first prove in 
Theorem~\ref{thm: error nonsmooth} that if~$u_0\in\Hilb$ then the
error is of order~$t_n^{-1}\Dt$, coinciding with the error
estimate~\eqref{eq: error heat} for the classical heat equation
when~$r=0$. Next we prove the special case~$r=\min(2,1/\nu)$ 
of \eqref{eq: main result} and then, in 
Theorem~\ref{thm: interp error}, deduce the general case by 
interpolation.
The paper concludes with Section~\ref{sec: numerical examples}, which 
presents the results of some computational experiments for a model 
1D~problem, as well as numerical evidence that the constant~$C$
in~\eqref{eq: main result} can be chosen independent of~$\nu$.
\section{Integral representations}\label{sec: integral rep}
Our error analysis relies on the Laplace transform
\[
\hat u(z)=\Lapl\{u(t)\}=\int_0^\infty e^{-zt}u(t)\,dt.
\]
A standard energy
argument~\cite{McLeanMustapha2009,McLeanThomeeWahlbin1996} shows that
$\|u(t)\|\le\|u_0\|$ so 
$\hat u(z)$ exists and is analytic in the right half-plane $\Re z>0$,
and since $\Lapl\{\partial_t^{1-\nu}u\}=z^{1-\nu}\hat u(z)$ 
and $\Lapl\{\partial_tu\}=z\hat u-u_0$, it
follows from~\eqref{eq: ivp um} that
$z\hat u_m+\lambda_m z^{1-\nu}\hat u_m=u_{0m}$, so
\[
\hat u_m(z)=\frac{u_{0m}}{z+\lambda_mz^{1-\nu}}.
\]
Thus, the Laplace inversion formula gives, for $n\ge1$~and any~$a>0$,
\[
u_m(t_n)=\frac{1}{2\pi i}\int_{a-i\infty}^{a+i\infty}e^{zt_n}
  \hat u_m(z)\,dz=\frac{u_{0m}}{2\pi i}\int_{a-i\infty}^{a+i\infty}
	\frac{e^{zt_n}}{1+\lambda_mz^{-\nu}}\,\frac{dz}{z},
\]
which, following a substitution, we may write as
\begin{equation}\label{eq: um(tn) a}
u_m(t_n)=\frac{u_{0m}}{2\pi i}\int_{a-i\infty}^{a+i\infty}
	\frac{e^{nz}}{1+\mu z^{-\nu}}\,\frac{dz}{z},
	\quad\text{where $\mu=\lambda_m\Dt^\nu$.}
\end{equation}
It follows using Jordan's lemma that
\begin{equation}\label{eq: um(tn)}
u_m(t_n)=\frac{u_{0m}}{2\pi i}\int_{-\infty}^{0^+}
	\frac{e^{nz}}{1+\mu z^{-\nu}}\,\frac{dz}{z}
	\quad\text{for $n\ge1$,}
\end{equation}
where the notation~$\int_{-\infty}^{0^+}$ indicates that the path of 
integration is a Hankel contour enclosing the negative real axis and
oriented counterclockwise.

Now consider the recurrence relation~\eqref{eq: recurrence} used
to compute the numerical solution.  The Fourier 
coefficients~$U^n_m=\iprod{U^n,\phi_m}$ satisfy
\begin{equation}\label{eq: Unm}
(1+\beta_0\Dt^\nu\lambda_m)U^n_m=U^{n-1}_m
	-\lambda_m\Dt^\nu\sum_{j=1}^{n-1}\beta_{n-j}U^j_m,
\end{equation}
and to obtain an integral representation of~$U^n_m$ analogous
to~\eqref{eq: um(tn)} we introduce the discrete-time Laplace transform
\begin{equation}\label{eq: discrete LT}
\Utilde(z)=\sum_{n=0}^\infty U^ne^{-nz}.
\end{equation}
Again, a standard energy argument shows that~$\|U^n\|\le\|u_0\|$ so
this series converges in the right half-plane~$\Re z>0$.
Multiplying \eqref{eq: Unm} by~$e^{-nz}$, summing over~$n$ and using 
the fact that the sum in~\eqref{eq: Unm} is a discrete convolution, 
we find that
\[
\bigl[1-e^{-z}+\mu\betatilde(z)\bigr]\Utilde_m(z)
	=\bigl[1+\mu\betatilde(z)\bigr]u_{0m},
\]
again with $\mu=\lambda_m\Dt^\nu$.
So, letting $\psi(z)=\betatilde(z)/(1-e^{-z})$,
\begin{equation}\label{eq: Utilde betatilde}
\Utilde_m(z)
	=u_{0m}\,\frac{1+\mu\betatilde(z)}{1-e^{-z}+\mu\betatilde(z)}
	=u_{0m}\,\frac{(1-e^{-z})^{-1}+\mu\psi(z)}{1+\mu\psi(z)}.
\end{equation}
For our subsequent analysis we now establish key properties of the 
function~$\psi(z)$.

Following appropriate shifts of the summation index, one finds that
\begin{equation}\label{eq: betatilde}
\betatilde(z)=\sum_{n=0}^\infty\beta_ne^{-nz}
	=(e^z-1)(1-e^{-z})\,\frac{\Li_{-\nu}(e^{-z})}{\Gamma(1+\nu)},
\end{equation}
where the polylogarithm~\cite{Lewin1981,Wood1992} is defined by
$\Li_p(z)=\sum_{n=1}^\infty z^n/n^p$ for $|z|<1$ and $p\in\C$;
thus,
\begin{equation}\label{eq: psi Li}
\psi(z)=(e^z-1)\,\frac{\Li_{-\nu}(e^{-z})}{\Gamma(1+\nu)}
  =\frac{1}{\Gamma(1+\nu)}\biggl(1+\sum_{n=1}^\infty
    \bigl[(n+1)^\nu-n^\nu\bigr]e^{-nz}\biggr).
\end{equation}
From the identity
\[
\frac{1}{n^p}=\frac{\Gamma(1-p)}{2\pi i}\int_{-\infty}^{0^+}
	e^{nw}w^{p-1}\,dw,
\]
we find, after interchanging the sum and integral, that
\begin{equation}\label{eq: Li Hankel}
\Li_p(e^{-z})=\frac{\Gamma(1-p)}{2\pi i}\int_{-\infty}^{0^+}
	\frac{w^{p-1}\,dw}{e^{z-w}-1}
\end{equation}
for $\Re z$ sufficiently large.  Thus, $\Li_p(e^{-z})$
possesses an analytic continuation to the strip~$-2\pi<\Im z<2\pi$
with a cut along the negative real axis~$(-\infty,0]$.  It follows
that $\psi(z)$ is analytic for~$z$ in the same cut strip, and
moreover
\begin{equation}\label{eq: bar psi}
\overline{\psi(z)}=\psi(\bar z)
\quad\text{and}\quad
\psi(z+2\pi i)=\psi(z).
\end{equation}

\begin{lemma}\label{lem: 1+mu psi}
If $|\Im z|\le\pi$ and $z\notin(-\infty,0]$, then
\begin{equation}\label{eq: psi integral}
\psi(z)=\frac{\sin\pi\nu}{\pi}\int_0^\infty
  \frac{s^{-\nu}}{1-e^{-z-s}}\,\frac{1-e^{-s}}{s}\,ds
\end{equation}
and $1+\mu\psi(z)\ne0$ for $0<\mu<\infty$.
\end{lemma}
\begin{proof}
Given~$z\notin(-\infty,0]$, we can choose a Hankel contour that does 
not enclose~$z$, and the formulae \eqref{eq: psi Li}~and \eqref{eq: Li 
Hankel} then imply that
\[
\psi(z)=\frac{e^z-1}{2\pi i}\int_{-\infty}^{0^+}
  \frac{w^{-\nu-1}\,dw}{e^{z-w}-1}.
\]
Since
\[
\frac{e^z-1}{e^{z-w}-1}=1+\frac{e^w-1}{1-e^{w-z}}
\quad\text{and}\quad
\int_{-\infty}^{0^+}w^{-\nu-1}\,dw=0,
\]
we have
\[
\psi(z)=\frac{1}{2\pi i}\int_{-\infty}^{0^+}
  \frac{w^{-\nu}}{1-e^{w-z}}\,\frac{e^w-1}{w}\,dw.
\]
Define contours along either side of the cut,
\begin{equation}\label{eq: Contour pm}
\Contour_\pm=\{\,se^{\pm i\pi}:\text{for $0<s<\infty$}\,\},
\end{equation}
so that $\arg(w)=\pm\pi$ if~$w\in\Contour_\pm$.  Noting that the 
integrand is $O(w^{-\nu})$ as~$w\to0$, we may collapse the Hankel 
contour into~$\Contour^+-\Contour^-$ to 
obtain~\eqref{eq: psi integral}.

The second part of the lemma amounts to showing 
that~$\psi(z)\notin(-\infty,0]$.
If $x\ge0$ and
$\alpha_n=e^{-xn}\bigl[(n+1)^\nu-n^\nu\bigr]$, then
\begin{equation}\label{eq: Re psi}
\psi(x+iy)=\frac{1}{\Gamma(1+\nu)}\biggl(1+\sum_{n=1}^\infty
  \alpha_n\cos ny-i\sum_{n=1}^\infty\alpha_n\sin ny\biggr).
\end{equation}
The sequence~$\alpha_n$ is
convex and tends to zero, so~\cite[pp.~183 and 228]{Zygmund1959}
\[
\Re\psi(x+iy)\ge\frac{1}{2\Gamma(1+\nu)}
\quad\text{and}\quad
\Im\psi(x+iy)<0
\quad\text{for $x\ge0$ and $0<y<\pi$,}
\]
and using \eqref{eq: bar psi} we find that
$\Im\psi(x\pm i\pi)=0$ for $-\infty<x<\infty$.
The polylogarithm satisfies~\cite[Equation~(3.1)]{Wood1992}
\[
\Im\Li_p(e^{-z})=\mp\frac{\pi s^{p-1}}{\Gamma(p)}
	\quad\text{if $z=se^{\pm i\pi}$ for $0<s<\infty$,}
\]
so, using the 
identity~$\Gamma(1+\nu)\Gamma(1-\nu)=\pi\nu/\sin\pi\nu$,
\begin{equation}\label{eq: Im psi cut}
\Im\psi(se^{\pm i\pi})=\mp(1-e^{-s})s^{-\nu-1}\sin\pi\nu,
\end{equation}
and in particular 
$\Im\psi(x+i0)<0$ but $\Im\psi(x-i0)>0$ for $-\infty<x<0$, whereas
$\Im\psi(x)=0$ for $0<x<\infty$.
Applying the strong maximum principle for harmonic functions,
we conclude that $\Im\psi(x+iy)\ne0$ if $0<|y|<\pi$.
We saw above that $\Re\psi(x+iy)>0$ if~$x\ge 0$, and 
by~\eqref{eq: psi integral},
\[
\psi(x\pm i\pi)=\frac{\sin\pi\nu}{\pi}\int_0^\infty 
  \,\frac{s^{-\nu}}{1+e^{-x-s}}\,\frac{1-e^{-s}}{s}\,ds>0
\]
for all real~$x$, which completes the proof.
\end{proof}

Since
\[
\frac{1}{2\pi i}\int_{a-i\pi}^{a+i\pi}e^{(n-j)z}\,dz
	=\delta_{nj}=\begin{cases}
	1,&\text{if $n=j$,}\\
	0,&\text{if $n\ne j$,}
\end{cases}
\]
we see from the definition~\eqref{eq: discrete LT} of~$\Utilde_m$,
after interchanging the sum and integral, that for any $a>0$,
\begin{equation}\label{eq: Unm integral}
U^n_m=\frac{1}{2\pi i}\int_{a-i\pi}^{a+i\pi}e^{nz}
	\Utilde_m(z)\,dz.
\end{equation}
Moreover, since
\[
\frac{(1-e^{-z})^{-1}+\mu\psi(z)}{1+\mu\psi(z)}
  =1+\frac{(1-e^{-z})^{-1}-1}{1+\mu\psi(z)}
  =1-\frac{1/(1-e^z)}{1+\mu\psi(z)},
\]
the formula~\eqref{eq: Utilde betatilde} for~$\Utilde_m(z)$ implies
that 
\begin{equation}\label{eq: Unm basic integral}
U^n_m=\frac{u_{0m}}{2\pi i}\int_{a-i\pi}^{a+i\pi}
  \frac{e^{nz}}{1+\mu\psi(z)}\,\frac{dz}{e^z-1}
\quad\text{for~$n\ge1$.}
\end{equation}
The next lemma describes the asymptotic behaviour of~$\psi$, and 
shows in particular that the integrands of \eqref{eq: um(tn) a}~and
\eqref{eq: Unm basic integral} are close for~$z$ near~$0$.
In~\eqref{eq: psi z->0}, $\zeta$ denotes the 
Riemann zeta function.

\begin{lemma}\label{lem: psi asymp}
The function~\eqref{eq: psi Li} satisfies
\begin{equation}\label{eq: psi z->0}
\psi(z)=z^{-\nu}+\tfrac12 z^{1-\nu}
  +\frac{\zeta(-\nu)}{\Gamma(1+\nu)}\,z+O(z^{2-\nu})
	\quad\text{as $z\to0$,}
\end{equation}
and
\begin{equation}\label{eq: psi Re z->-oo}
\psi(z)
	=\frac{\sin\pi\nu}{\pi\nu}\,(i\pi-z)^{-\nu}+O(z^{-\nu-1})
	\quad\text{as $\Re(z)\to-\infty$, with $0<\Im z<\pi$.}
\end{equation}
\end{lemma}
\begin{proof}
Flajolet~\cite[Theorem~1]{Flajolet1999} shows that
\begin{equation}\label{eq: Lip z->0}
\Li_p(e^{-z})\sim\Gamma(1-p)z^{p-1}+\sum_{k=0}^\infty
	(-1)^k\zeta(p-k)\,\frac{z^k}{k!}\quad\text{as $z\to0$,}
\end{equation}
and \eqref{eq: psi z->0} follows because $e^z-1=z+\tfrac12 
z^2+O(z^3)$ as~$z\to0$.
The results of Ford~\cite[Equation~(17), p.~226]{Ford1960} imply that
\begin{equation}\label{eq: Li_p z->-oo}
\Li_p(e^{-z})=-\frac{(i\pi-z)^p}{\Gamma(1+p)}+O(z^{p-1})
	\quad\text{as $\Re z\to-\infty$,}
\end{equation}
(see also Wood~\cite[Equation~(11.2)]{Wood1992}) which, in 
combination with the 
identity $\Gamma(1+\nu)\Gamma(1-\nu)=\pi\nu/\sin\pi\nu$, implies
\eqref{eq: psi Re z->-oo}.
\end{proof}

\begin{figure}
\begin{center}
\includegraphics[scale=0.5,trim=0 80 0 50]{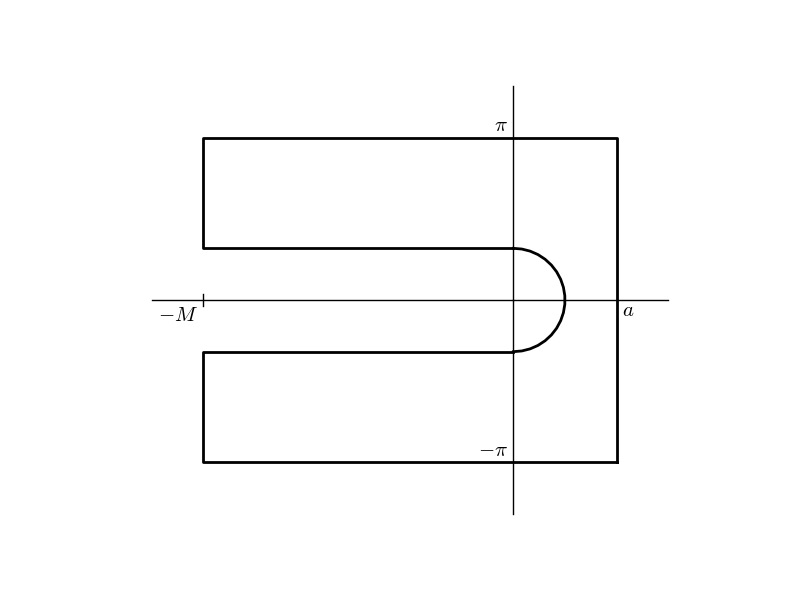}
\caption{The integration contour~$\Contour(a,M)$.}
\label{fig: contour}
\end{center}
\end{figure}

The formula for~$U^n_m$ in the next theorem matches \eqref{eq: um(tn)} 
for~$u_m(t_n)$.

\begin{theorem}\label{thm: Unm Hankel}
The solution of~\eqref{eq: Unm} admits the integral representation
\begin{equation}\label{eq: Unm psi}
U^n_m=\frac{u_{0m}}{2\pi i}\int_{-\infty}^{0^+}
  \frac{e^{nz}}{1+\mu\psi(z)}\,\frac{dz}{e^z-1}\quad
  \text{for $n\ge1$,}
\end{equation}
where the Hankel contour remains inside the strip~$-\pi<\Im z<\pi$.
\end{theorem}
\begin{proof}
By Lemma~\ref{lem: 1+mu psi},
the integrand from~\eqref{eq: Unm basic integral} is
analytic for~$z$ inside the contour~$\Contour(a,M)$ shown in
Figure~\ref{fig: contour}.  The contributions along $\Im z=\pm\pi$
cancel in view of the second part of~\eqref{eq: bar psi}.
Using~\eqref{eq: psi Re z->-oo}, if $\Re z\to-\infty$ then
\[
\frac{1/(e^z-1)}{1+\mu\psi(z)}\\
  \sim-\biggl(1+\mu\,\frac{\sin\pi\nu}{\pi\nu}\,(i\pi-z)^{-\nu}
  \biggr)^{-1}
\sim-1+\mu\,\frac{\sin\pi\nu}{\pi\nu}\, (i\pi-z)^{-\nu},
\]
so the contributions along~$\Re z=-M$ are $O(e^{-nM})$
as~$M\to\infty$, implying the desired formula for~$U^n_m$. 
\end{proof}

Together, \eqref{eq: um(tn)}~and \eqref{eq: Unm psi} imply that
the error formula~\eqref{eq: Un-u(tn)} holds, with
\begin{equation}\label{eq: delta}
\delta^n(\mu)=\frac{1}{2\pi i}\int_{-\infty}^{0^+}e^{nz}
	\biggl(\frac{1}{1+\mu\psi(z)}\,\frac{z}{e^z-1}
	-\frac{1}{1+\mu z^{-\nu}}
	\biggr)\frac{dz}{z}
\end{equation}
for~$0<\mu<\infty$, and with~$\delta^n(0)=0$ because if~$\lambda_m=0$ 
then $u_m(t_n)=u_{0m}=U^n_m$ for all~$n$.
The following estimate for~$\delta^n(\mu)$ is the key to proving 
our error estimates, but the lengthy proof is deferred until
Section~\ref{sec: technical}.

\begin{theorem}\label{thm: delta^n(mu)}
Let $0<\nu<1$. The sequence~\eqref{eq: delta} satisfies
\[
|\delta^n(\mu)|
  \le Cn^{-1}\min\bigl((\mu n^\nu)^2,(\mu n^\nu)^{-1}\bigr)
  \quad\text{for $n=1$, $2$, $3$, \dots and $0<\mu<\infty$.}
\]
\end{theorem}
\begin{proof}
Follows from Theorems \ref{thm: rho>=1}~and \ref{thm: rho<=1}.
\end{proof}

We remark that in the limiting case~$\nu\to1$, when our method 
reduces to the classical implicit Euler 
scheme~\eqref{eq: implicit Euler} for the heat 
equation~\eqref{eq: heat eqn}, it is readily seen that the error 
representation~\eqref{eq: Un-u(tn)} holds with
$
\delta^n(\mu)=(1+\mu)^{-n}-e^{-n\mu},
$
and that 
$0\le\delta^n(\mu)\le Cn^{-1}\min\bigl((\mu n)^2,(\mu n)^{-1}\bigr)$,
consistent with Theorem~\ref{thm: delta^n(mu)}.
\section{Error estimates}\label{sec: error}

We begin this section with the basic error bound that applies even 
when no smoothness is assumed for the initial data.

\begin{theorem}\label{thm: error nonsmooth}
For any~$u_0\in\Hilb$, the solutions of \eqref{eq: ivp}~and
\eqref{eq: DG} satisfy
\[
\|U^n-u(t_n)\|\le Ct_n^{-1}\Dt\|u_0\|
  \quad\text{for $n=1$, $2$, $3$, \dots.}
\]
\end{theorem}
\begin{proof}
Theorem~\ref{thm: delta^n(mu)} implies that 
$|\delta^n(\mu)|\le Cn^{-1}$ uniformly for~$0<\mu<\infty$, and
since the $\phi_m$ are orthonormal, we see from~\eqref{eq: Un-u(tn)}
that
\begin{equation}\label{eq: Parseval}
\|U^n-u(t_n)\|^2=\sum_{m=1}^\infty
  \bigl[\delta^n(\lambda_m\Dt^\nu) u_{0m}\bigr]^2
	\le (Cn^{-1})^2\sum_{m=1}^\infty u_{0m}^2
	=\bigl(Cn^{-1}\|u_0\|\bigr)^2.
\end{equation}
The estimate follows after recalling that $t_n=n\Dt$ so
$n^{-1}=t_n^{-1}\Dt$.
\end{proof}

For smoother initial data, the error bound exhibits a less severe 
deterioration as~$t_n$ approaches zero.

\begin{lemma}\label{lem: error smooth}
Consider the solutions of \eqref{eq: ivp}~and \eqref{eq: DG}.
\begin{enumerate}
\item If $0<\nu\le1/2$ and $A^2u_0\in\Hilb$, then 
\[
\|U^n-u(t_n)\|\le Ct_n^{2\nu-1}\Dt\|A^2u_0\|\le C\Dt^{2\nu}\|A^2u_0\|.
\]
\item If $1/2\le\nu<1$ and $A^{1/\nu}u_0\in\Hilb$, then
\[
\|U^n-u(t_n)\|\le C\Dt\|A^{1/\nu}u_0\|.
\]
\end{enumerate}
\end{lemma}
\begin{proof}
In the first case, since $\lambda_m\Dt^\nu n^\nu=\lambda_m t_n^\nu$,
\begin{align*}
|\delta^n(\lambda_m\Dt^\nu)|&\le Ct_n^{-1}\Dt\,
\min\bigl((\lambda_mt_n^\nu)^2,(\lambda_mt_n^\nu)^{-1}\bigr)\\
  &=C t_n^{2\nu-1}\Dt\,\lambda_m^2
    \min\bigl(1,(\lambda_mt_n^\nu)^{-3}\bigr)
      \le Ct_n^{2\nu-1}\Dt\,\lambda_m^2,
\end{align*}
so by \eqref{eq: ||A^rv||}~and \eqref{eq: Parseval},
\[
\|U^n-u(t_n)\|^2\le\sum_{m=1}^\infty
  \bigl(Ct_n^{2\nu-1}\Dt\,\lambda_m^2u_{0m}\bigr)^2\\
  =\bigl(Ct_n^{2\nu-1}\Dt\,\|A^2u_0\|\bigr)^2,
\]
with $t_n^{2\nu-1}\Dt=n^{2\nu-1}\Dt^{2\nu}\le\Dt^{2\nu}$.
The second case follows in a similar fashion, because
$n^{-1}=\Dt\,\lambda_m^{1/\nu}(\lambda_mt_n^\nu)^{-1/\nu}$ 
implies that
\[
|\delta^n(\lambda_m\Dt^\nu)|\le C\Dt\,\lambda_m^{1/\nu}
\min\bigl((\lambda_mt_n^\nu)^{2-1/\nu},
  (\lambda_mt_n^\nu)^{-1-1/\nu}\bigr)
    \le C\Dt\,\lambda_m^{1/\nu}.
\]
\end{proof}

We are now ready to prove our main result.

\begin{theorem}\label{thm: interp error}
The solutions of \eqref{eq: ivp}~and \eqref{eq: DG} satisfy
\[
\|U^n-u(t_n)\|\le C t_n^{r\nu-1}\Dt\|A^ru_0\|
	\quad\text{for $0\le r\le\min(2,1/\nu)$.}
\]
\end{theorem}
\begin{proof}
If $0<\nu\le1/2$ and $0<\theta<1$, then by interpolation
\[
\|U^n-u(t_n)\|
	\le C\bigl(t_n^{-1}\Dt\bigr)^{1-\theta}
	\bigl(t_n^{2\nu-1}\Dt\bigr)^\theta 
	\|A^{2\theta}u_0\|
	=Ct_n^{2\nu\theta-1}\Dt\|A^{2\theta}u_0\|,
\]
and the estimate follows by putting $r=2\theta$.
Similarly, if $1/2\le\nu<1$, then
\[
\|U^n-u(t_n)\|\le C\bigl(t_n^{-1}\Dt\bigr)^{1-\theta}
	\Dt^\theta\|A^{\theta/\nu}u_0\|
	=Ct_n^{\theta-1}\Dt\|A^{\theta/\nu}u_0\|,
\]
and the estimate follows by putting $r=\theta/\nu$.  
\end{proof}

\section{Technical proofs}\label{sec: technical}

It remains to prove Theorem~\ref{thm: delta^n(mu)}.  In this section
only, $C$ always denotes an absolute constant and we use subscripts
in cases where the constant might depend on some parameters; for 
instance $C_\nu$ may depend on the fractional diffusion 
exponent~$\nu$. 

Since the integrand of~\eqref{eq: delta} is $O(z^{\nu-1})$ as~$z\to0$, 
we may collapse the Hankel contour onto~$\Contour_+-\Contour_-$, for
$\Contour_\pm$ given by~\eqref{eq: Contour pm}.  In this way, 
defining  
\[
\psi_\pm(s)=\psi(se^{\pm i\pi})\quad\text{for $0<s<\infty$,}
\]
we find that
\begin{multline*}
\int_{\Contour_\pm}e^{nz}
	\biggl(\frac{1}{1+\mu\psi(z)}\,\frac{z}{e^z-1}
	-\frac{1}{1+\mu z^{-\nu}}
	\biggr)\frac{dz}{z}\\
  =\int_0^\infty e^{-ns}\biggl(
  \frac{1}{1+\mu\psi_\pm(s)}\,\frac{s}{1-e^{-s}}
  -\frac{1}{1+\mu s^{-\nu}e^{\mp i\pi\nu}}\biggr)\,\frac{ds}{s}.
\end{multline*}
By \eqref{eq: bar psi}~and \eqref{eq: Im psi cut},
\begin{equation}\label{eq: bar psi pm}
\psi_-(s)=\overline{\psi_+(s)}
\quad\text{and}\quad
\Im\psi_\pm(s)=\mp(1-e^{-s})s^{-\nu-1}\sin\pi\nu,
\end{equation}
so
\[
\frac{1}{1+\mu\psi_+(s)}-\frac{1}{1+\mu\psi_-(s)}
  =\frac{2i\mu\Im\psi_-(s)}{|1+\mu\psi_\pm(s)|^2}
  =\frac{2i\mu s^{-\nu}\sin\pi\nu}{|1+\mu\psi_\pm(s)|^2}
  \,\frac{1-e^{-s}}{s},
\]
and similarly,
\[
\frac{1}{1+\mu s^{-\nu}e^{-i\pi\nu}}
-\frac{1}{1+\mu s^{-\nu}e^{i\pi\nu}}
  =\frac{2i\mu s^{-\nu}\sin\pi\nu}{|1+\mu s^{-\nu}e^{\mp i\pi\nu}|^2}.
\]
Thus, the representation~\eqref{eq: delta} implies
\begin{equation}\label{eq: delta cut}
\delta^n(\mu)=\frac{\sin\pi\nu}{\pi}\int_0^\infty e^{-ns}\mu s^{-\nu}
\biggl(\frac{1}{|1+\mu\psi_+(s)|^2}
  -\frac{1}{|1+\mu s^{-\nu}e^{-i\pi\nu}|^2}\biggr)\,\frac{ds}{s}.
\end{equation}
We will estimate this integral with the help of the following 
sequence of lemmas.  

\begin{lemma}\label{lem: |1+X...|}
If $X\ge 0$ then
$|1+Xe^{\pm i\pi\nu}|^{-2}\le(1-\nu)^{-2}(1+X^2)^{-1}$.
\end{lemma}
\begin{proof}
Since $0\le2X/(1+X^2)\le1$,
\[
\frac{|1+Xe^{\pm i\pi\nu}|^2}{1+X^2}
  =\frac{|e^{\mp i\pi\nu}+X|^2}{1+X^2}=1+\frac{2X}{1+X^2}\,\cos\pi\nu
  \ge\min(1,1+\cos\pi\nu),
\]
and the result follows because 
$1+\cos\pi\nu=2\cos^2(\pi\nu/2)\ge2(1-\nu)^2$.
\end{proof}

\begin{lemma}\label{lem: |1+mu...|}
If $\mu\ge0$ and $s>0$, then
$|1+\mu\psi_\pm(s)|^{-2}\le C_\nu(1+\mu^2 s^{-2\nu})^{-1}$.
\end{lemma}
\begin{proof}
Lemma~\ref{lem: psi asymp} implies that
\begin{equation}\label{eq: psi s->0}
\psi_\pm(s)=e^{\mp i\pi\nu}(s^{-\nu}-\tfrac12s^{1-\nu})
	-\frac{\zeta(-\nu)}{\Gamma(1+\nu)}\,s+O(s^{2-\nu})
	\quad\text{as $s\to0$}
\end{equation}
and
\begin{equation}\label{eq: psi s->oo}
\psi_\pm(s)
	=\frac{\sin\pi\nu}{\pi\nu}\,s^{-\nu}+O(s^{-\nu-1})
	\quad\text{as $s\to\infty$.}
\end{equation}
Thus, if we define $\phi(s)=s^\nu\psi_+(s)$ for $0<s<\infty$, with
\begin{equation}\label{eq: phi 0 oo}
\phi(0)=e^{-i\pi\nu}\quad\text{and}\quad
\phi(\infty)=\frac{\sin\pi\nu}{\pi\nu},
\end{equation}
then $\phi$ is continuous on the one-point 
compactification~$[0,\infty]$ of the closed half-line~$[0,\infty)$. 
Put $X=\mu s^{-\nu}$ and define
\[
f(s,X)=\frac{|1+\mu\psi_+(s)|^2}{1+X^2}=\frac{|1+X\phi(s)|^2}{1+X^2}
\]
for $0\le s\le\infty$ and $0\le X<\infty$, with 
$f(s,\infty)=|\phi(s)|^2$, so that $f$ is continuous on the compact
topological space~$[0,\infty]\times[0,\infty]$.  
It therefore suffices to prove that $f$ is 
strictly positive everywhere.  By~\eqref{eq: bar psi pm},
\begin{equation}\label{eq: Im phi}
\Im\phi(s)=-\frac{1-e^{-s}}{s}\sin\pi\nu<0
  \quad\text{for $0<s<\infty$},
\end{equation}
and $\Im\phi(0)=-\sin\pi\nu<0$ by~\eqref{eq: phi 0 oo}, so
$|1+X\phi(s)|^2\ge[X\Im\phi(s)]^2>0$ for $0\le s<\infty$~and 
$0<X<\infty$.  Moreover, $|1+X\phi(\infty)|^2\ge1$ because 
$\phi(\infty)$ is real and positive, and $f(s,0)=1$ 
for~$0\le s\le\infty$. Finally, \eqref{eq: phi 0 oo}~and
\eqref{eq: Im phi} imply that $f(s,\infty)=|\phi(s)|^2>0$ 
for~$0\le s\le\infty$.
\end{proof}

\begin{lemma}\label{lem: diff squares}
For $\mu\ge0$~and $s>0$,
\begin{multline*}
|1+\mu s^{-\nu}e^{\mp i\pi\nu}|^2-|1+\mu\psi_{\pm}(s)|^2\\
	=\mu B_+(s)\bigl(1+\mu s^{-\nu}e^{i\pi\nu}\bigr)
      +\mu B_-(s)\bigl(1+\mu\psi_+(s)\bigr)
	=\mu B_1(s)+\mu^2B_2(s),
\end{multline*}
where $B_\pm(s)=s^{-\nu}e^{\mp i\pi\nu}-\psi_\pm(s)$ and
\begin{align*}
B_1(s)&=B_+(s)+B_-(s)
	=2\bigl(s^{-\nu}\cos\pi\nu-\Re\psi_\pm(s)\bigr),\\
B_2(s)&=B_+(s)s^{-\nu}e^{i\pi\nu}+B_-(s)\psi_+(s)
	=s^{-2\nu}-\psi_+(s)\psi_-(s).
\end{align*}
\end{lemma}
\begin{proof}
Put $a=\mu s^{-\nu}e^{\mp i\pi\nu}$ and $b=\mu\psi_\pm$
in the identities
\begin{align*}
|1+a|^2-|1+b|^2&=(a-b)(1+\bar a)+(\bar a-\bar b)(1+b)\\
	&=(a-b)+(\bar a-\bar b)+(a\bar a-b\bar b).
\end{align*}
\end{proof}

Notice that $B_1$~and $B_2$ are real, whereas
$B_-(s)=\overline{B_+(s)}$.

\begin{lemma}\label{lem: B asymp}
As $s\to0$,
\[
B_\pm(s)=O(s^{1-\nu}),\quad
B_1(s)=s^{1-\nu}\cos\pi\nu+O(s),\quad
B_2(s)=s^{1-2\nu}+O(s^{1-\nu}),
\]
and as~$s\to\infty$,
\[
B_\pm(s)=O(s^{-\nu}),\quad B_1(s)=O(s^{-\nu}),\quad
B_2(s)=O(s^{-2\nu}).
\]
\end{lemma}
\begin{proof}
Follows using \eqref{eq: psi s->0}~and \eqref{eq: psi s->oo}.
\end{proof}

We are now ready to prove the easier half of 
Theorem~\ref{thm: delta^n(mu)}.  

\begin{theorem}\label{thm: rho>=1}
For $0<\mu<\infty$ and $n=1$, $2$, $3$, \dots,
the sequence~\eqref{eq: delta} satisfies
\[
|\delta^n(\mu)|\le C_\nu n^{-1}\rho^{-1}\quad
	\text{if $\rho=\mu n^\nu$.}
\]
\end{theorem}
\begin{proof}
From \eqref{eq: delta cut}~and Lemma~\ref{lem: diff squares}, we
see that $\delta^n(\mu)$ equals
\[
\frac{\sin\pi\nu}{\pi}\int_0^\infty e^{-ns}\mu s^{-\nu}
\frac{\mu B_+(s)\bigl(1+\mu s^{-\nu}e^{i\pi\nu}\bigr)
+\mu B_-(s)\bigl(1+\mu\psi_+(s)\bigr)}%
{|1+\mu s^{-\nu}e^{i\pi\nu}|^2|1+\mu\psi_+(s)|^2}\,\frac{ds}{s},
\]
and thus, by Lemmas \ref{lem: |1+X...|}~and \ref{lem: |1+mu...|},
\[
|\delta^n(\mu)|\le C_\nu\int_0^\infty e^{-ns}\mu s^{-\nu}\,
	\frac{\mu|B_\pm(s)|}{(1+\mu^2s^{-2\nu})^{3/2}}\,\frac{ds}{s}.
\]
Lemma~\ref{lem: B asymp} implies that
$|B_\pm(s)|\le C_\nu\min\bigl(s^{1-\nu},s^{-\nu}\bigr) 
=C_\nu s^{-\nu}\min(s,1)$, so 
\[
|\delta^n(\mu)|\le C_\nu\int_0^\infty g_n(s,\mu)\,ds
\quad\text{where}\quad
g_n(s,\mu)=e^{-ns}\mu^2\,
  \frac{s^{-2\nu-1}\min(s,1)}{(1+\mu^2s^{-2\nu})^{3/2}}.
\]
The estimate for~$\delta^n(\mu)$ follows because
\[
\int_0^1 g_n(s,\mu)\,ds\le\int_0^1e^{-ns}\,\frac{s^\nu}{\mu}\,ds
  =\frac{n^{-1-\nu}}{\mu}\int_0^ne^{-s}s^\nu\,ds
  \le\frac{\Gamma(1+\nu)}{n\rho}
\]
and
\[
\int_1^\infty g_n(s,\mu)\,ds
  \le\int_1^\infty e^{-ns}\frac{s^{\nu-1}}{\mu}\,ds
  \le\int_1^\infty\frac{e^{-ns}}{\mu}\,ds
  =\frac{n^\nu}{\rho}\,\frac{e^{-n}}{n}
  \le\frac{C}{n\rho}.
\]
\end{proof}

Establishing the behaviour of~$\delta^n(\mu)$ when~$\rho=\mu n^\nu$ 
is small turns out to be more delicate, and relies on three additional 
lemmas.

\begin{lemma}\label{lem: <=3}
If $0\le\nu\le1/2$ then $x^\nu\int_x^1s^{-3\nu}\,ds\le3$ 
for~$0<x\le1$.
\end{lemma}

\begin{proof}
Let $f(x)=x^\nu\int_x^1s^{-3\nu}\,ds$. 
If $0<\nu<1/3$ then 
\begin{equation}\label{eq: f'(x*)}
\text{$f'(x)>0$ for $0<x<x^*$}
\quad\text{and}\quad
\text{$f'(x)<0$ for $x^*<x<1$,} 
\end{equation}
where~$x^*=[\nu/(1-2\nu)]^{1/(1-3\nu)}<1$.
Since $f'(x)=\nu x^{-1}f(x)-x^{-2\nu}$, 
\[
f(x)\le f(x^*)=\frac{(x^*)^{1-2\nu}}{\nu}=\frac{(x^*)^\nu}{1-2\nu}
	\le 3.
\]
If $\nu=1/3$, then $f(x)=x^{1/3}\log x^{-1}$ and \eqref{eq: f'(x*)}
holds with~$x^*=e^{-3}$, implying that $f(x)\le f(x^*)=3e^{-1}\le 3$.
If $1/3<\nu<1/2$, then~\eqref{eq: f'(x*)} holds 
with~$x^*=[(1-2\nu)/\nu]^{1/(3\nu-1)}<1$ and again
$f(x)\le f(x^*)=(x^*)^{1-2\nu}/\nu\le3$.  Finally, if $\nu=0$ then
$f(x)=1-x\le1$, and
if $\nu=1/2$ then $f(x)=2(1-x^{1/2})\le2$.
\end{proof}

\begin{lemma}\label{lem: <=3 nu>=1/2}
If $1/2\le\nu\le1$ then $x^{\nu-1}\int_1^x s^{1-3\nu}\,ds\le3$
for $1\le x<\infty$.
\end{lemma}

\begin{proof}
Make the substitutions $x'=x^{-1}$, $s'=s^{-1}$, $\nu'=1-\nu$
in Lemma~\ref{lem: <=3}.
\end{proof}

\begin{figure}
\begin{center}
\includegraphics[scale=0.5,trim=0 60 0 30]{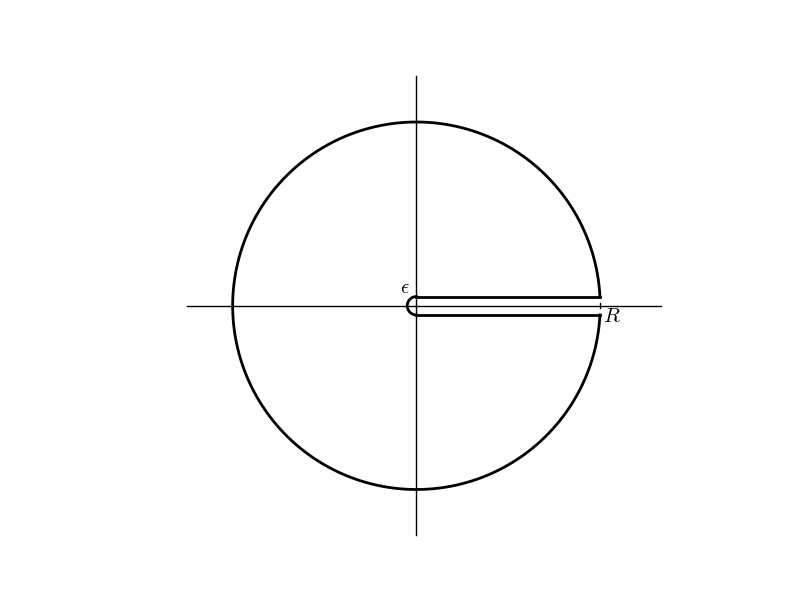} 
\end{center}
\caption{The contour~$\Contour(\epsilon,R)$ used in the proof of
Lemma~\ref{lem: integral=0}.}
\label{fig: path}
\end{figure}

\begin{lemma}\label{lem: integral=0}
If $1/2<\nu<1$ then
\[
\int_0^\infty
\frac{s^{-2\nu}\cos\pi\nu+s^{-3\nu}}{|1+s^{-\nu}e^{i\pi\nu}|^4}\,ds
=\int_0^\infty
\frac{s^\nu+s^{2\nu}\cos\pi\nu}{|s^\nu+e^{i\pi\nu}|^4}\,ds
	=0.
\]
\end{lemma}
\begin{proof}
Let $p=-\cos\pi\nu$ so that $0<p<1$.  Making the
substitution~$x=s^\nu$, we see that the integral equals $\nu^{-1}I$,
where 
\[
I=\int_0^\infty f(x)\,dx
\quad\text{and}\quad
f(x)=\frac{1-px}{(x^2-2px+1)^2}\,x^{1/\nu}.
\]
We consider the analytic continuation of~$f$ to the cut 
plane~$\C\setminus[0,\infty)$, and note that 
$z^2-2pz+1=(z-\alpha_+)(z-\alpha_-)$ where
$\alpha_\pm=p\pm iq=e^{i\pi(1\mp\nu)}$ and
$q=\sqrt{1-p^2}=\sin\pi\nu$. Thus, $f$ has double poles at
$z=\alpha_+$~and at~$\alpha_-$.  Moreover, since $1<1/\nu<2$ we
see that $f(z)=o(|z|^{-1})$ as~$|z|\to\infty$, and that $f(z)=O(|z|)$
as~$|z|\to0$.  After integrating around the
contour~$\Contour(\epsilon,R)$ shown in Figure~\ref{fig: path} and
sending $\epsilon\to0^+$ and $R\to\infty$, we conclude that
\[
\frac{1-e^{i2\pi/\nu}}{2\pi i}\,I
	=\res_{z=\alpha_+}f(z)+\res_{z=\alpha_-}f(z).
\]
Since 
$(z-\alpha_{\pm})^2f(z)=(1-pz)z^{1/\nu}/(z-\alpha_\mp)^2$
and $\alpha_+^{1/\nu}=-e^{i\pi/\nu}=\alpha_-^{1/\nu}$,
\[
\res_{z=\alpha_\pm}f(z)
	=\lim_{z\to\alpha_\pm}\frac{d}{dz}(z-\alpha_\pm)^2f(z)
	=\frac{d}{dz}\,\frac{(1-pz)z^{1/\nu}}{(z-\alpha_\mp)^2}
	\bigg|_{z=\alpha_\pm}
	=\mp i\,\frac{1-\nu}{\nu}\,\frac{e^{i\pi/\nu}}{4q},
\]
showing that the residues cancel, and therefore $I=0$ because
$e^{i2\pi/\nu}\ne1$.
\end{proof}

Our final result for this section completes the proof of 
Theorem~\ref{thm: delta^n(mu)}, and hence of the error estimates of 
Section~\ref{sec: error}.  

\begin{theorem}\label{thm: rho<=1}
For $0<\mu<\infty$ and $n=1$, $2$, $3$, \dots,
the sequence~\eqref{eq: delta} satisfies
\[
|\delta^n(\mu)|\le C_\nu n^{-1}\rho^2\quad
	\text{if $\rho=\mu n^\nu\le1$.}
\]
\end{theorem}
\begin{proof}
By Lemma~\ref{lem: B asymp},
\[
\mu B_1(s)+\mu^2 B_2(s)
	=s\bigl(\mu s^{-\nu}\cos\pi\nu+(\mu s^{-\nu})^2
	+O(\mu+\mu^2 s^{-\nu})\bigr)\quad\text{as $s\to0^+$,}
\]
and $\mu B_1(s)+\mu^2 B_2(s)=O(\mu s^{-\nu}+\mu^2s^{-2\nu})$
as $s\to\infty$, so \eqref{eq: delta cut} implies that
\begin{align*}
|\delta^n(\mu)|&=\biggl|\frac{\sin\pi\nu}{\pi} \int_0^\infty 
	e^{-ns}\mu s^{-\nu}
	\frac{\mu B_1(s)+\mu^2B_2(s)}{|1+\mu s^{-\nu}e^{i\pi\nu}|^2
|1+\mu\psi_+(s)|^2}\,\frac{ds}{s}\biggr|\\
	&\le\frac{\sin\pi\nu}{\pi}\bigl(
	|I_1|+C_\nu I_2+C_\nu I_3\bigr),
\end{align*}
where, using Lemmas \ref{lem: |1+X...|}~and \ref{lem: |1+mu...|},
\begin{gather*}
I_1=\int_0^1 e^{-ns}\mu s^{-\nu}\,
\frac{\mu s^{-\nu}\cos\pi\nu+(\mu s^{-\nu})^2}%
{|1+\mu s^{-\nu}e^{i\pi\nu}|^2|1+\mu\psi_+(s)|^2}\,ds,\\
I_2=\int_0^1e^{-ns}\mu s^{-\nu}\,
\frac{\mu+\mu^2 s^{-\nu}}{(1+\mu^2 s^{-2\nu})^2}\,ds,\quad
I_3=\int_1^\infty e^{-ns}\mu s^{-\nu}\,
\frac{\mu s^{-\nu}+\mu^2 s^{-2\nu}}%
{(1+\mu^2 s^{-2\nu})^2}\,\frac{ds}{s}.
\end{gather*}
Put $f(x)=(x+x^2)/(1+x^2)^2$ so that
\[
I_2=\mu\int_0^1 e^{-ns}f(\mu s^{-\nu})\,ds
	= n^{-1-\nu}\rho\int_0^n e^{-s}f(\rho s^{-\nu})\,ds.
\]
Since $f(x)\le\min(2x,x^{-2})$ we have
$f(\rho s^{-\nu})\le C\min(\rho^{-2}s^{2\nu}, \rho s^{-\nu})$
and thus
\begin{align*}
n^{1+\nu}\rho^{-1}I_2
	&\le C\rho^{-2}\int_0^{\rho^{1/\nu}}e^{-s}s^{2\nu}\,ds
	+C\rho\int_{\rho^{1/\nu}}^n e^{-s}s^{-\nu}\,ds\\
	&\le C\int_0^{\rho^{1/\nu}}e^{-s}\,ds
	+C\rho\int_{\rho^{1/\nu}}^1s^{-\nu}\,ds
	+C\rho\int_1^\infty e^{-s}\,ds\\
	&\le C\rho^{1/\nu}+C(1-\nu)^{-1}\rho+C\rho
	\le C(1-\nu)^{-1}\rho+C\rho^{1/\nu}\le C_\nu\rho,
\end{align*}
implying $I_2\le C_\nu n^{-1-\nu}\rho^2 \le C_\nu n^{-1}\rho^2$.  
Noting 
that $\mu=\rho n^{-\nu}\le1$, we have
\[
I_3\le \int_1^\infty e^{-ns}\mu^2 s^{-2\nu-1}\,ds
	\le\mu^2\int_1^\infty e^{-ns}\,ds=\mu^2\,\frac{e^{-n}}{n}
	\le n^{-1}\mu^2 =n^{-1-2\nu}\rho^2,
\]
and therefore $I_3\le n^{-1}\rho^2$.

It remains to estimate~$I_1$.  First consider the case~$0<\nu<1/2$, in
which $\cos\pi\nu>0$.  Put $g(x)=(x^2\cos\pi\nu+x^3)/(1+x^2)^2$, so 
that	
\[
I_1\le C_\nu\int_0^1e^{-ns}g(\mu s^{-\nu})\,ds
	=C_\nu n^{-1}\int_0^ne^{-s}g(\rho s^{-\nu})\,ds.
\]
Since $g(x)\le\min(2x^2,x^{-2}\cos\pi\nu+x^{-1})$ we
have 
\[
g(\rho s^{-\nu})\le C\min\bigl(\rho^{-1}s^\nu, 
\rho^2s^{-2\nu}\cos\pi\nu+\rho^3s^{-3\nu}\bigr)
\]
and hence $\int_0^ne^{-s}g(\rho s^{-\nu})\,ds$ is bounded by
\begin{align*}
&C\rho^{-1}\int_0^{\rho^{1/\nu}}s^\nu\,ds
	+C\rho^2\cos\pi\nu\int_{\rho^{1/\nu}}^n
	e^{-s}s^{-2\nu}\,ds
	+C\rho^3\int_{\rho^{1/\nu}}^n e^{-s}s^{-3\nu}\,ds\\
	&\le C\rho^{1/\nu}
	+C\rho^2\int_{\rho^{1/\nu}}^1(1-2\nu)s^{-2\nu}\,ds
	+C\rho^3\int_{\rho^{1/\nu}}^1s^{-3\nu}\,ds
		+C\rho^2\int_1^\infty e^{-s}\,ds.
\end{align*}
Applying Lemma~\ref{lem: <=3} with~$x=\rho^{1/\nu}$ and noting that
$1/\nu>2$, it follows that 
$\int_0^ne^{-s}g(\rho s^{-\nu})\,ds\le 
C\bigl(\rho^{1/\nu}+\rho^2\bigr)$ and hence
$I_1\le C_\nu n^{-1}\rho^2$.

If $\nu=1/2$, then $\cos\pi\nu=0$ and the argument above again shows
that $I_1\le C_\nu n^{-1}\rho^2$.  Thus,
assume now that $1/2<\nu<1$ and note $\cos\pi\nu<0$.  Since
\begin{multline*}
\frac{e^{-ns}}%
{|1+\mu s^{-\nu}e^{i\pi\nu}|^2|1+\mu\psi_+(s)|^2}
  =\frac{1}{|1+\mu s^{-\nu}e^{i\pi\nu}|^4}\\
  -\frac{1-e^{-ns}}{|1+\mu s^{-\nu}e^{i\pi\nu}|^2|1+\mu\psi_+(s)|^2}
+\frac{|1+\mu s^{-\nu}e^{i\pi\nu}|^2-|1+\mu\psi_+(s)|^2}%
{|1+\mu s^{-\nu}e^{i\pi\nu}|^4|1+\mu\psi_+(s)|^2}
\end{multline*}
and, by Lemma~\ref{lem: integral=0},
\begin{align*}
\int_0^1\frac{(\mu s^{-\nu})^2\cos\pi\nu+(\mu s^{-\nu})^3}%
{|1+\mu s^{-\mu}e^{i\pi\nu}|^4}\,ds
  &=\mu^{1/\nu}\int_0^{\mu^{-1/\nu}}
\frac{s^{-2\nu}\cos\pi\nu+s^{-3\nu}}{|1+s^{-\nu}e^{i\pi\nu}|^4}\,ds\\
  &=-\mu^{1/\nu}\int_{\mu^{-1/\nu}}^\infty
\frac{s^{-2\nu}\cos\pi\nu+s^{-3\nu}}{|1+s^{-\nu}e^{i\pi\nu}|^4}\,ds,
\end{align*}
we have
\begin{equation}\label{eq: I1 tricky} 
|I_1|\le C_\nu\bigl(J_1+J_2+J_3\bigr),
\end{equation}
where
\begin{align*}
J_1&=\mu^{1/\nu}\int_{\mu^{-1/\nu}}^\infty
\frac{s^{-2\nu}\cos\pi\nu+s^{-3\nu}}{(1+\mu^2s^{-\nu})^2}\,ds,\\
J_2&=\int_0^1(1-e^{-ns})\,
\frac{(\mu s^{-\nu})^2|\cos\pi\nu|+(\mu s^{-\nu})^3}%
{(1+\mu^2s^{-2\nu})^2}\,ds,\\
J_3&=\int_0^1
\bigl(|1+\mu s^{-\nu}e^{i\pi\nu}|^2-|1+\mu\psi_+(s)|^2\bigr)
\frac{(\mu s^{-\nu})^2|\cos\pi\nu|+(\mu s^{-\nu})^3}%
{(1+\mu^2s^{-2\nu})^3}\,ds.
\end{align*}
First, because $\mu^{1/\nu}=n^{-1}\rho^{1/\nu}$~and
$|\cos\pi\nu|=\sin\pi(\nu-\tfrac12)\le\pi(\nu-\tfrac12)$,  
\begin{align*}
J_1&\le Cn^{-1}\rho^{1/\nu}\int_{n\rho^{-1/\nu}}^\infty\bigl(
  (2\nu-1)s^{-2\nu}+s^{-3\nu}\bigr)\,ds\\
   &\le Cn^{-1}\rho^{1/\nu}\bigl[(n\rho^{-1/\nu})^{1-2\nu}
   +(n\rho^{-1/\nu})^{1-3\nu}\bigr]\\
   &=Cn^{-2\nu}\rho^2+Cn^{-3\nu}\rho^3\le Cn^{-1}\rho^2.
\end{align*}
Second, since $1-e^{-x}\le x$ and $\mu^{-1/\nu}=n\rho^{-1/\nu}\ge1$,
we see that $n\rho^{-1/\nu}J_2$ equals
\begin{align*}
\int_0^{n\rho^{-1/\nu}}&(1-e^{-\rho^{1/\nu}s})
\frac{s^{-2\nu}|\cos\pi\nu|+s^{-3\nu}}{(1+s^{-2\nu})^2}\,ds
\le C\int_0^1\frac{(1-e^{-\rho^{1/\nu}s})s^{-3\nu}}%
{(1+s^{-2\nu})^2}\,ds\\
	&\qquad{}
	+C\int_1^{n\rho^{-1/\nu}}(1-e^{-\rho^{1/\nu}s})
	\bigl(s^{-2\nu}(\nu-\tfrac12)+s^{-3\nu}\bigr)\,ds\\
	&\le C\rho^{1/\nu}\int_0^1 s^{\nu+1}\,ds
	+C\int_{\rho^{1/\nu}}^n(1-e^{-s})\bigl(
	\rho^3s^{-3\nu}+(\nu-\tfrac12)\rho^2s^{-2\nu}\bigr)\,ds.
\end{align*}
Since $\rho^3s^{-3\nu}\le\rho^2s^{-2\nu}$ for~$s\ge\rho^{1/\nu}$, the 
last integral is bounded by
\begin{multline*}
\int_{\rho^{1/\nu}}^1 2\rho^2s^{1-2\nu}\,ds
+C\int_1^n (2\nu-1)\bigl(\rho^3s^{-3\nu}+\rho^2s^{-2\nu}\bigr)\,ds\\
	  \le C\int_{\rho^{1/\nu}}^1\rho^2s^{-1}\,ds
	  +C\rho^3+C\rho^2
\le C\rho^{3-1/\nu}+C\rho^2\log\rho^{-1/\nu},	  
\end{multline*}
and thus
\[
J_2\le Cn^{-1}\rho^{1/\nu}\bigl(\rho^{1/\nu}+C\rho^{3-1/\nu}
  +\nu^{-1}\rho^2\log\rho^{-1}\bigr)
\le C_\nu n^{-1}\rho^2.
\]

Third, by Lemmas \ref{lem: diff squares}~and \ref{lem: B asymp},
\begin{align*}
J_3&\le\int_0^1\bigl(\mu s^{1-\nu}+\mu^2s^{1-2\mu}\bigr)\,
  \frac{(\mu s^{-\nu})^2+(\mu s^{-\nu})^3}{(1+\mu s^{-\nu})^3}\,ds\\
  &=\mu^{1+1/\nu}\int_0^{\mu^{-1/\nu}}
  \frac{s(s^{-\nu}+s^{-2\nu})(s^{-2\nu}+s^{-3\nu})}{(1+s^{-2\nu})^3}
  \,ds\\
  &\le(\rho n^{-\nu})^{1+1/\nu}\biggl(\int_0^1 s^{1+\nu}\,ds
    +\int_1^{n\rho^{-1/\nu}}s^{1-3\nu}\,ds\biggr),
\end{align*}
and applying Lemma~\ref{lem: <=3 nu>=1/2} with~$x=n\rho^{-1/\nu}$
gives
$\int_1^{n\rho^{-1/\nu}}s^{1-3\nu}\,ds\le3(n\rho^{-1/\nu})^{1-\nu}$
so $J_3\le Cn^{-\nu-1}\rho^{1+1/\nu}(1+n^{1-\nu}\rho^{1-1/\nu})
\le C(n^{-\nu-1}\rho^{1+1/\nu}+n^{-2\nu}\rho^2)
\le Cn^{-1}\rho^2$.
Inserting the foregoing estimates for $J_1$, $J_2$~and 
$J_3$ into~\eqref{eq: I1 tricky} gives the desired estimate
$|I_1|\le Cn^{-1}\rho^2$, which completes the proof.
\end{proof}

\section{Numerical example}\label{sec: numerical examples}
We consider a 1D example in which $u=u(x,t)$ 
satisfies~\eqref{eq: ivp} with $Au=-(\kappa u_x)_x$ 
for~$x\in\Omega=(-1,1)$, subject to homogeneous Dirichlet boundary 
conditions~$u(\pm1,t)=0$ for~$0<t\le1$.  We 
choose~$\kappa=4/\pi^2$ so the orthonormal eigenfunctions and 
corresponding eigenvalues of~$A$ are
\[
\phi_m(x)=\sin\frac{m\pi}{2}(x+1)
\quad\text{and}\quad
\lambda_m=m^2\quad\text{for $m=1$, $2$, $3$, \dots.}
\]
For our initial data we choose simply the constant 
function~$u_0(x)=\pi/4$, which has the Fourier sine coefficients
\[
u_{0m}=\iprod{u_0,\phi_m}=\begin{cases}
	m^{-1},&m=1, 3, 5, \ldots,\\
	0,&m=2, 4, 6, \ldots.
\end{cases}
\]
Although infinitely differentiable, the function~$u_0$ is  
``non-smooth'' because it fails to satisfy the boundary conditions, 
and as a result the solution~$u(x,t)$ is discontinuous at~$x=\pm1$
when~$t=0$. In fact, if~$r<1/4$ then
\[
\|A^ru_0\|^2=\sum_{m=1}^\infty\bigl(\lambda_m^ru_{0m}\bigr)^2
  =\sum_{j=1}^\infty(2j-1)^{4r-1}
    \le\frac{C}{1-4r},
\]
but if $r\ge1/4$ then $u_0\notin D(A^r)$.

\begin{figure}
\begin{center}
\begin{minipage}[b]{0.48\textwidth}
\begin{center}
\includegraphics[width=\textwidth,%
trim=90 30 0 20]{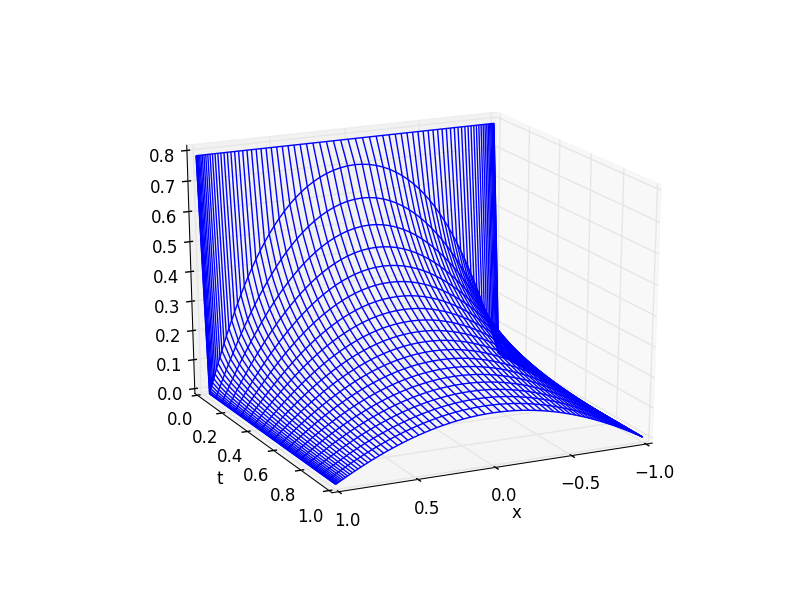}
\end{center}
\end{minipage}
\begin{minipage}[b]{0.48\textwidth}
\begin{center}
\includegraphics[width=\textwidth,%
trim=90 30 0 20]{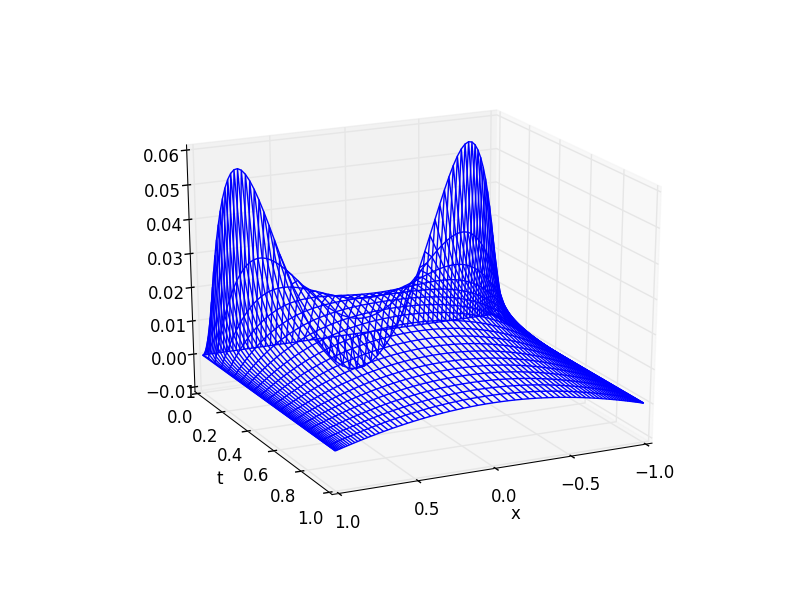}
\end{center}
\end{minipage}
\end{center}
\caption{Reference solution (left) and error (right).}
\label{fig: example}
\end{figure}

Using a closed form expression for~$\hat u(x,z)$, we construct a 
reference solution by applying a spectrally accurate numerical 
method~\cite{McLeanThomee2010} for inversion of the Laplace 
transform.  To compute the discrete-time solution~$U^n$ we
discretize also in space using piecewise linear finite elements on a
fixed nonuniform mesh with $M$~subintervals.  In view of the 
discontinuity in the solution when~$t=0$, we concentrate the 
spatial grid points near~$x=\pm1$, but always use a constant 
timestep~$\Dt=1/N$. 

Figure~\ref{fig: example} shows the reference solution and the error
in the case~$\nu=0.75$ using $N=20$ time steps and $M=80$ spatial 
subintervals. 
As expected, the error is largest at the first time level~$t_1$ and 
then decays as~$t$ increases.
We put 
$r=\tfrac{1}{4}-\epsilon$ where $\epsilon^{-1}=\max(4,\log t_n^{-1})$,
so that $t_n^{-\epsilon}\le C$ and, by 
Theorem~\ref{thm: interp error},
\[
\|U^n-u(t_n)\|\le Ct_n^{\nu/4-1}\Dt\,\sqrt{\max(1,\log t_n^{-1})}
  \quad\text{for $0<t_n\le1$.}
\]
Thus, ignoring the logarithm and putting $\nu=3/4$, we expect to
observe errors of order~$t_n^{-13/16}\Dt$.

\begin{figure}
\begin{center}
\includegraphics[width=0.8\textwidth,%
trim=0 40 0 20]{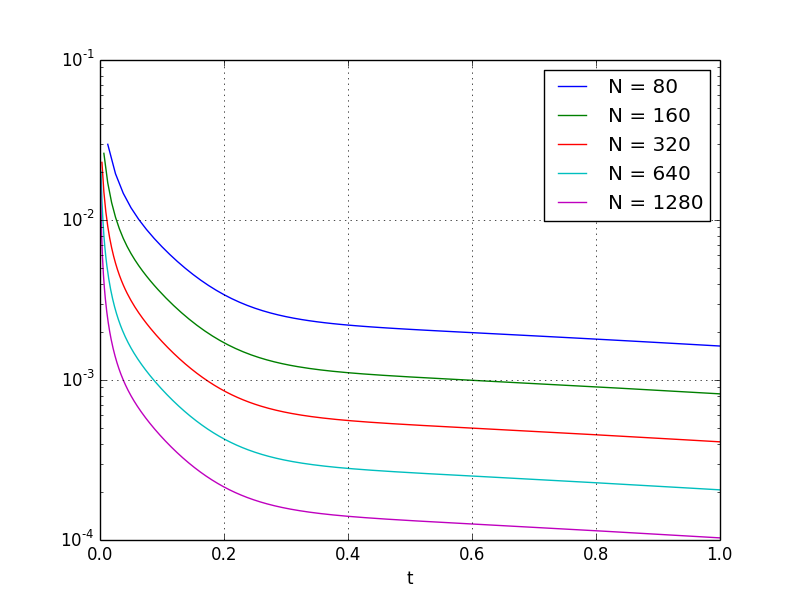}
\end{center}
\caption{The error $\|U^n-u(t_n)\|$ as a function of~$t_n$.}
\label{fig: log errors}
\end{figure}

\begin{table}
\begin{center}
\renewcommand{\arraystretch}{1.2}
\begin{tabular}{r|cc|cc|cc}
\multicolumn{1}{c|}{$N$}&
\multicolumn{2}{c|}{$\alpha=0.6$}&
\multicolumn{2}{c|}{$\alpha=0.7$}&
\multicolumn{2}{c}{$\alpha=13/16$}\\
\hline
  80&  2.14e-03&       &  1.48e-03&       &  1.16e-03&      \\
 160&  1.24e-03&0.788  &  7.94e-04&0.894  &  5.91e-04&0.978 \\
 320&  7.20e-04&0.787  &  4.29e-04&0.888  &  2.98e-04&0.988 \\
 640&  4.17e-04&0.787  &  2.32e-04&0.887  &  1.50e-04&0.992 \\
1280&  2.42e-04&0.787  &  1.25e-04&0.887  &  7.53e-05&0.993 
\end{tabular}
\end{center}
\caption{Weighted errors and observed convergence 
rates from~\eqref{eq: wt error}.}
\label{tab: wt errors}
\end{table}

\begin{figure}
\begin{center}
\includegraphics[width=0.8\textwidth,%
trim=0 40 0 20]{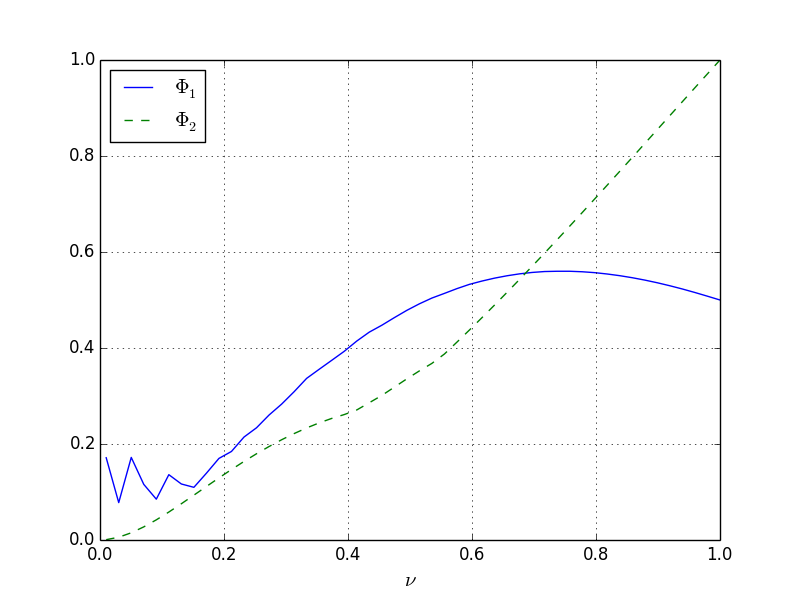}
\end{center}
\caption{The functions $\Phi_1$~and $\Phi_2$ from \eqref{eq: Phi}.}
\label{fig: Phi}
\end{figure}

Figure~\ref{fig: log errors} shows how the error varies with~$t_n$
for a sequence of solutions obtained by successively doubling~$N$
(and hence halving $\Dt$), using a log scale. (The same spatial mesh 
with~$M=1000$ subintervals was used in all cases.)  
Table~\ref{tab: wt errors} provides an alternative view of this data, 
listing the weighted error and its associated convergence rate,
\begin{equation}\label{eq: wt error}
E_N=\max_{1\le t_n\le 1/2}t_n^\alpha\|U^n-u(t_n)\|
\quad\text{and}\quad
\rho_N=\log_2(E_N/E_{N/2}),
\end{equation}
so that if $E_N$ decays like $N^{-\rho}=\Dt^\rho$ then 
$\rho\approx\rho_N$.
As expected, $\rho_N\approx1$ when $\alpha=13/16=0.8125$, but 
the rate deteriorates for smaller values of~$\alpha$.  

Our analysis in Section~\ref{sec: technical} does not reveal how the 
constant in Theorem~\ref{thm: delta^n(mu)} depends on the fractional 
diffusion exponent~$\nu$, because the proof of 
Lemma~\ref{lem: |1+mu...|} is not constructive.  The
factor~$(1-\nu)^{-2}$ in the estimate of Lemma~\ref{lem: |1+X...|}
raises the question of whether the DG error becomes large if~$\nu$ is 
very close to~$1$.  We therefore investigated numerically the values
of
\begin{equation}\label{eq: Phi}
\begin{aligned}
\Phi_1(\nu)&=\sup_{0<\mu<\infty}\max_{n^\nu\le\mu^{-1}}
  n^{1-2\nu}\mu^{-2}\delta^n(\mu),\\
\Phi_2(\nu)&=\sup_{0<\mu<\infty}\sup_{n^\nu\ge\mu^{-1}}
  n^{1+\nu}\mu\delta^n(\mu),
\end{aligned}
\end{equation}
since $C=\max\bigl(\Phi_1(\nu),\Phi_2(\nu)\bigr)$ is the best possible
constant in Theorem~\ref{thm: delta^n(mu)}.  Figure~\ref{fig: Phi}
shows approximations of the graphs of $\Phi_1$~and $\Phi_2$, obtained
by restricting $\mu$ to the discrete values~$2^j$ for~$-18\le j\le20$,
and resticting $n$ to the range~$1\le n\le200$.  We solved 
\eqref{eq: ivp um}~and \eqref{eq: Unm} with $u_{0m}=1=U^0_m$ and
$\lambda_m=\mu/\Dt^\nu$ to compute $\delta^n(\mu)=U^n_m-u_m(t_n)$.
The evaluation of~$\Phi_1(\nu)$ is problematic for~$\nu$ near zero 
because our values for~$u_m(t_n)$ are not sufficiently accurate, but 
it seems reasonable to conjecture that $C\le1$ for all~$\nu$.

\paragraph{Acknowledgement} We thank Peter Brown for help with the 
proof of Lemma~\ref{lem: integral=0}.
\bibliographystyle{plain}
\bibliography{nonsmoothrefs}
\end{document}